\documentclass[10pt]{amsart}

\usepackage{amssymb}
\usepackage{graphicx}

\usepackage{hyperref}

\def\N{{\mathbb{N}}}

\def\R{{\mathbb{R}}}

\def\C{{\mathbb{C}}}
\newcommand{\ssm}{\smallsetminus}
\newcommand{\overbar}[1]{\mkern 1.5mu\overline{\mkern-1.5mu#1\mkern-1.5mu}\mkern 1.5mu}

\numberwithin{example}{subsection}

\newtheorem{proposition}{Proposition}
\newtheorem{theorem}{Theorem}
\newtheorem{lemma}{Lemma}

\title{Detecting conjugacy stability of subgroups in certain classes of groups}

\author{Isabel Fern\'andez Mart\'inez}
\address{} 
\email{isabel10791@gmail.com}

\author{Denis Serbin}
\address{Department of Mathematical Sciences, Stevens Institute of Technology, 1 Castle Point on Hudson, Hoboken, NJ 07030, USA} 
\email{d.e.serbin@gmail.com}

\keywords{Conjugacy stable, conjugacy-closed, Frattini embedded}

\subjclass[2010]{20F65, 20F67, 20F10}

\date{}

\begin{document}

\begin{abstract}
In this paper we consider the {\em conjugacy stability} property of subgroups and provide effective procedures to solve the problem in several classes of groups. In particular, we start with free groups, that is, we give an effective procedure to find out if a finitely generated subgroup of a free group is conjugacy stable. Then we further generalize this result to quasi-convex subgroups of torsion-free hyperbolic groups and finitely generated subgroups of limit groups.
\end{abstract}

\maketitle

\tableofcontents

\section{Introduction}
\label{se:intro}

A subgroup $H$ of a group $G$ is called {\em conjugacy stable} if every two elements from $H$ conjugate in $G$ are also conjugate in $H$ itself.  It is obvious that the property of being conjugacy stable is transitive: if $H$ is conjugacy stable in $G$ and $K$ is conjugacy stable in $H$, then $K$ is conjugacy stable in $G$. Further, if $H \mathrel{\unlhd} G$ is conjugacy stable, then it is a transitively normal subgroup of $G$, that is, every subgroup of $H$ is also normal in $G$.

The property of subgroups to be conjugacy stable was studied in linear groups (under the name {\em conjugacy-closed}). It is known that, if $k$ is a subfield of a field $K$, then the general linear group $GL_n(k)$ over $k$ is a conjugacy stable subgroup of the general linear group $GL_n(K)$ over $K$ (not true if $k$ is only a subring of $K$). Next, the orthogonal group over reals $O_n(\R)$ is a conjugacy stable subgroup of $GL_n(\R)$. Further, the unitary group of complex matrices $U_n(\C)$ is a conjugacy stable subgroup of $GL_n(\C)$. Eventually, Brauer's permutation lemma (see \cite{Brauer:1941}) can be reformulated as follows: the subgroup of permutation matrices in $GL_n(k)$, where $k$ is any field of characteristic zero, is conjugacy stable in $GL_n(k)$.

This property is also known in the context of symmetric groups. Namely, let $A$ be a subset of a set $B$. Then, naturally $Sym(A) \leqslant Sym(B)$ and if either $A$, or $B \ssm A$ is finite, then $Sym(A)$ is conjugacy stable in $Sym(B)$.

Last but not least, conjugacy stable subgroups are known under the name of {\em Frattini embedded} subgroups (see \cite{Thompson:1980}). It is easy to see that if $H$ is a conjugacy stable subgroup of a group $G$ and the conjugacy problem is solvable in $G$, then it is also solvable in $H$. This observation was used in \cite{Ol'shanskii_Sapir:2004}  to obtain interesting embedding results.

More recently, the conjugacy stability property was investigated in the case of parabolic subgroups of Artin-Tits groups of spherical type in \cite{Calvez_CisnerosDeLaCruz_Cumplido:2020}.

\smallskip

The goal of this paper is to study this interesting subgroup property from an algorithmic viewpoint. Namely, we concentrate on the {\em conjugacy stability problem}: find an effective procedure that, given a group $G$ and its subgroup $H$, determines if $H$ is conjugacy stable in $G$. We solve the problem for several classes of groups, in particular, free groups, torsion-free hyperbolic groups, and limit groups.

\section{General approach to the problem}
\label{se:general}

Let $G$ be a group and let $H \leqslant G$. Everywhere in the text below we use the convention $H^g = g^{-1} H g$.

Recall the definition given in the introduction. $H$ is called {\em conjugacy stable} in $G$ if for every $g \in G \ssm H$ such that $g^{-1} u g = v$ for some $u, v \in H$, there exists $h \in H$ such that $h^{-1} u h = v$. 

First of all, notice that if there exist $u, v \in H$ such that $g^{-1} u g = v$ for some $g \in G \ssm H$, then $H^g \cap H \neq 1$. Now, existence of $h \in H$ such that $h^{-1} u h = v$ can be reformulated as shown in the lemma below.

\begin{lemma} 
\label{le:general_1}
Let $G$ be a group and let $H \leqslant G$. Then $H$ is conjugacy stable if and only if for every $g \in G \ssm H$ such that $H^g \cap H \neq 1$ and every non-trivial $u \in H^{g^{-1}} \cap H$, the intersection $H \cap C_G(u) g$ is non-empty.
\end{lemma}
\begin{proof}
In order to determine if $H$ is conjugacy stable, for every $g \in G \ssm H$ such that $g^{-1} u g = v$ for some $u, v \in H$, one has to find out if there exists $g_0 \in H$ such that $g_0^{-1} u g_0 = v$.

Observe that if for a given $g$ such $g_0$ exists then $g^{-1} u g = g_0^{-1} u g_0$ and $[g_0 g^{-1}, u] = 1$, so $g_0 g^{-1} \in C_G(u)$ and $g_0 \in C_G(u) g$. Hence, if $H$ is conjugacy stable, then for every $g \in G \ssm H$, such that $H^g \cap H \neq 1$, and every $u \in H^{g^{-1}} \cap H$, we have $H \cap C_G(u) g \neq \varnothing$.

The converse is also true. Indeed, if $g^{-1} u g = v$ for some $u, v \in H$, then $u \in H \cap H^{g^{-1}}$ and by assumption there exists $h \in H \cap C_G(u) g$. But then $h = a g, [a, u] = 1$ and
$$h^{-1} u h = g^{-1} (a^{-1} u a) g  = g^{-1} u g = v,$$
that is, $H$ is conjugacy stable.
\end{proof}

The above lemma makes the process of checking is $H$ is conjugacy stable more concrete: for every $u \in H$ that can be conjugated back into $H$ by some $g \in G \ssm H$, it is enough to check if $H \cap C_G(u) g$ is non-empty. Unfortunately, this simplification alone does not lead to an effective procedure since there can be infinitely many elements $g \in G \ssm H$ such that $H^g \cap H \neq 1$.

\smallskip

In order to deal with this issue we introduce the following subgroup property. We say that $H$ satisfies the {\em bounded non-trivial conjugate intersection} ({\em BNTCI} for short) property if there are finitely many elements $g_1, \ldots, g_n \in G$ such that if $H^g \cap H \neq 1$ for some $g \in G$, then there exists $i \in [1, n]$ such that $g \in H g_i H$. Let us call the elements $g_1, \ldots, g_n$ {\em BNTCI representatives in $G$ by $H$} and denote 
$$\mathcal{R}_G(H) = \{g_1, \ldots, g_n\}.$$

Notice that the BNTCI property has connections with the {\em bounded packing} property introduced in \cite{HruskaWise:2009}: a subgroup $H$ of a finitely generated group $G$ has {\em bounded packing} in $G$ if for each constant $D$, there is a number $N = N(G, H, D)$ so that for any collection of $N$ distinct cosets $g H$ in $G$, at least two are separated by a distance of at least $D$. It is known (see \cite{GitikMitraRipsSageev:1998}, \cite{HruskaWise:2009}) that quasi-convex subgroups of word-hyperbolic groups have bounded packing. Now, if $G$ is a torsion-free word-hyperbolic group and $H$ is a quasi-convex subgroup of $G$, then bounded packing of $H$ in $G$ implies the BNTCI property for $H$.

A quantitative variant of the bounded packing property was introduced in \cite[Section 6.4]{KMW:2014} for quasi-convex subgroups of automatic groups. Namely, let $G$ with a finite set of semigroup  generators $X$ be an automatic group with respect to a regular language $L \subseteq X^\ast$ and let $\nu : \N \to \N$ be a non-decreasing function. We say that the regular structure $(G, L)$ satisfies property $BP_\nu$ if, whenever $H$ and $K$ are $L$-quasi-convex subgroups of $G$ with constant of $L$-quasi-convexity $k$, if $K, g_1 H, \ldots, g_n H$ are pairwise distinct and if $K \cap \left(\bigcap_i H^{g_i}\right)$ is infinite, then $K$ as well as each $g_i H$ meets a ball of radius $\nu(k)$ in the Cayley graph of $G$. This more technical condition, if satisfied, implies the BNTCI property for $L$-quasi-convex subgroups of $G$ (see, \cite[Proposition 6.7]{KMW:2014}).

Now let us see how the BNTCI property helps with the conjugacy stability problem.

\begin{lemma} 
\label{le:general_2}
Let $G$ be a group and let $H \leqslant G$ satisfy the BNTCI property. Then $H$ is conjugacy stable if and only if for every $g_i \in \mathcal{R}_G(H)$ and every non-trivial $u \in H^{g_i^{-1}} \cap H$, the intersection $H \cap C_G(u) g_i$ is non-empty.
\end{lemma}
\begin{proof} 
If $H$ is conjugacy stable, then, by Lemma \ref{le:general_1}, for every $g \in G$ and every non-trivial $u \in H^{g^{-1}} \cap H$, the intersection $H \cap C_G(u) g$ is non-empty. In particular, this holds for every $g_i \in \mathcal{R}_G(H)$.

\smallskip

Let us prove the converse. Suppose $g \in G \ssm H$ is such that $H^g \cap H \neq 1$. By the BNTCI property there exist $h, f \in H$ such that $g = h g_i f$ for some $g_i \in \mathcal{R}_G(H)$. Then 
$$g H g^{-1} \cap H = (h g_i f) H (f^{-1} g_i^{-1} h^{-1}) \cap H = h (g_i H g_i^{-1} \cap H) h^{-1}$$
and for every $u \in g H g^{-1} \cap H$ we have that $h^{-1} u h \in g_i H g_i^{-1} \cap H$. Next, 
$$H \cap C_G(u) g = H \cap C_G(u) (h g_i f).$$
Observe that $H \cap C_G(u) g \neq \varnothing$ if and only if there exists $w \in H$ such that $w = a g,\ a \in C_G(u)$. Hence, $w = a (h g_i f) = h (h^{-1} a h) g_i f$, or
$$h^{-1} w f^{-1} = (h^{-1} a h) g_i \in C_G(h^{-1} u h) g_i,$$
where $h^{-1} w f^{-1} \in H$. Hence, $H \cap C_G(u) g \neq \varnothing$ if and only if $H \cap C_G(h^{-1} u h) g_i \neq \varnothing$, and the latter holds by assumption. Thus, the condition from Lemma \ref{le:general_1} is satisfied and $H$ is conjugacy stable.
\end{proof}

As follows from Lemma \ref{le:general_2}, the BNTCI property of $H$ really simplifies checking if $H$ is conjugacy stable and we will use it whenever the property is available.

Observe also that if $H \leqslant G$ satisfies the BNTCI property and $H^g \cap H$ is non-trivial, then there exist $h, f \in $ such that $g = h g_i f$ for some $g_i \in \mathcal{R}_G(H)$ and we have
$$H^g \cap H = H^{h g_i f} \cap H = f^{-1} (H^{g_i} \cap H) f.$$
That is, there is a finite list of subgroups $\mathcal{J}_G(H)$ of $G$ of the form $H^{g_i} \cap H,\ g_i \in \mathcal{R}_G(H)$ such that for every $g \in G$, if $H^g \cap H \neq 1$, then $H^g \cap H$ is conjugate to some $J \in \mathcal{J}_G(H)$ by an element from $H$.

\smallskip

Finally, notice that the mere existence of BNTCI representatives in $G$ by $H$ is not enough for algorithmic solution of the conjugacy stability problem: one has to be able to find $\mathcal{R}_G(H)$ effectively. Moreover, there are some other technical difficulties such as checking if the intersection $H \cap C_G(u) g_i$ is non-empty for every $u$ in the subgroup $H^{g_i^{-1}} \cap H$, which can be infinite.

Below we consider several classes of groups, where the method of solving the conjugacy stability problem based on Lemma \ref{le:general_1} and Lemma \ref{le:general_2} works.

\section{Effective solution of the conjugacy stability problem}
\label{se:applications}

In this section we apply the strategy outlined in Section \ref{se:general}. Namely, we consider several classes of groups, where finitely generated subgroups either have the BNTCI property (free and torsion-free word-hyperbolic groups), or ``almost'' have it (limit groups), and solve the conjugacy stability problem in those classes of groups.

\subsection{Free groups}
\label{se:free}

Let $F = F(X)$ be a free group on a finite alphabet $X$. Note that from \cite[Proposition 9.7]{Kapovich_Myasnikov:2002} and its proof it follows that every finitely generated subgroup $H$ of $F$ satisfies the BNTCI property. Moreover, a set of BNTCI representatives $\mathcal{R}_F(H)$ in $G$ by $H$ can be found effectively.

Let us fix a finitely generated subgroup $H \leqslant F$ and a set of BNTCI representatives $\mathcal{R}_F(H)$.

Suppose $g \in \mathcal{R}_F(H)$ and $u \in g H g^{-1} \cap H$. Observe that if $C_F(u) = \langle u \rangle$, that is, $C_F(u) = C_H(u)$, which happens only if $u$ is not a proper power in $F$, then 
$$H \cap C_F(u) g = H \cap \langle u \rangle g \subseteq H \cap H g = \varnothing.$$
It follows that if for some $g \in \mathcal{R}_F(H)$, the intersection $g H g^{-1} \cap H$ contains an element which is not a proper power in $F$, then $H$ fails to be conjugacy stable.

Suppose, on the other hand, that for $g \in \mathcal{R}_F(H)$, the intersection $g H g^{-1} \cap H$ contains only elements which are proper powers in $F$. Take any $h_1, h_2, h_3 \in g H g^{-1} \cap H$ such that $h_1 h_2 = h_3$. The elements $h_1, h_2, h_3$ are proper powers in $F$, so there exist $a, b, c \in F$ such that 
$$h_1 = a^k,\ h_2 = b^l,\ h_3 = c^m,\ k, l, m \geqslant 2$$
and we have
$$a^k b^l = c^m.$$
From \cite{Lyndon_Schutzenberger:1962} it follows that $a, b$ and $c$ commute with each other, hence, $h_1, h_2$, and $h_3$ must also commute and $g H g^{-1} \cap H$ is cyclic.

From the observation above the following proposition follows.

\begin{proposition}
\label{pr:free_algorithm}
Let $F = F(X)$ be a free group on a finite alphabet $X$. Then there exists an effective procedure to check if any finitely generated subgroup $H \leqslant F$ is conjugacy stable or not.
\end{proposition}
\begin{proof}
Let $H$ be a finitely generated subgroup of $F$. 

\smallskip

\textbf{Step 1.} Compute $\mathcal{R}_F(H)$. It is finite and can be found effectively by \cite[Proposition 9.7]{Kapovich_Myasnikov:2002}.

\textbf{Step 2.} For every $g \in \mathcal{R}_F(H)$, find generators of the intersection $g H g^{-1} \cap H$. This can be done algorithmically by \cite[Corollary 9.5]{Kapovich_Myasnikov:2002}.

\textbf{Step 3.} Check if every $g H g^{-1} \cap H$ is cyclic. This can be done effectively as follows: construct a folded graph $\Gamma(g, H)$ representing $g H g^{-1} \cap H$, find a loop in $\Gamma(g, H)$ labeled by some $h_g \in H$, and check if $g H g^{-1} \cap H = \langle h_g \rangle$. All these operations can be done effectively (see \cite{Kapovich_Myasnikov:2002} for details).

If at least one intersection $g H g^{-1} \cap H$ is not cyclic, or if it is cyclic, but the generator is a root element in $F$ (this can be checked effectively), then stop. In this case $H$ is not conjugacy stable.

\textbf{Step 4.} Assuming that for every $g \in \mathcal{R}_F(H)$ we have 
$$g H g^{-1} \cap H = \langle h_g^{n_g} \rangle,$$
where $h_g$ is a root element in $F$ and $n_g \geqslant 2$, check if $H \cap C_F(h_g) g = H \cap \langle h_g \rangle g$ is non-empty. This can be done effectively (again, see \cite{Kapovich_Myasnikov:2002} for details). 

If $H \cap \langle h_g \rangle g$ is non-empty for every $g \in \mathcal{R}_F(H)$ then $H$ is conjugacy stable, otherwise it is not conjugcy stable.
\end{proof}

\subsection{Torsion-free hyperbolic groups}
\label{se:hyperbolic}

Recall that a finitely generated group $G$ is called {\em (word) hyperbolic} if its Cayley graph $\Gamma(G, S)$ with respect to some finite generating set $S$ satisfies the following property: there exists a constant $\delta > 0$, called a constant of hyperbolicity of $G$, such that every geodesic triangle $\Delta(x, y, z)$ with vertices $x, y, z \in \Gamma(G, S)$ is $\delta$-thin meaning that each side of $\Delta(x, y, z)$ lies inside the union of $\delta$-neighborhoods of the other two sides. If $S$ is fixed and such constant $\delta$ exists then we also call $G$ {\em $\delta$-hyperbolic}. 

Hyperbolic groups were introduced by Gromov in \cite{Gromov:1987} and now it is a well-studied class of groups (we refer the reader to \cite{Gromov:1987}, \cite{ABCFLMSS:1990}, \cite{Ghys_delaHarpe:1991}, \cite{Bridson_Haefliger:1999} for basic facts about hyperbolic groups). Notice that there is an algorithm that computes a constant of hyperbolicity $\delta$ given a  finite presentation of a group $G$ - the algorithm stops if $G$ is hyperbolic (see \cite{Papasoglu:1996}, \cite{EpsteinHolt:2001}).

For the rest of this section we fix a torsion-free $\delta$-hyperbolic group $G$ with a finite generating set $S$. A subgroup $H$ of $G$ is called {\em quasi-convex} if there exists a constant $k > 0$, called a quasi-convexity constant of $H$, such that every geodesic in the Cayley graph $\Gamma(G, S)$ connecting a pair of points that belong to $H$, lies inside the $k$-neighborhood of $H$. Notice that every quasi-convex subgroup $H$ is finitely generated and there exists an algorithm that computes a quasi-convexity constant of $H$ (see \cite{Kapovich:1996}). 

Quasi-convex subgroups have a lot of nice properties, in particular, many algorithmic problems for subgroups of hyperbolic groups can be solved only for quasi-convex subgroups. Hence, we only concentrate on the case, when $H \leqslant G$ is quasi-convex in $G$ with constant of quasi-convexity $k$. 

By \cite[Proposition 6.9]{KMW:2014}, the regular structure $(G, L)$, where $L$ is the language of all geodesics in $G$, satisfies property $BP_\nu$ (see the definition in Section \ref{se:general}) for the function $\nu(k) = k + 2\delta$, where $\delta$ is the hyperbolicity constant for $G$. We can assume that $\delta$ can be computed given the finite presentation of $G$, so the function $\nu$ is computable. Now, by \cite[Proposition 6.7(1)]{KMW:2014}, if $H$ has constant of quasi-convexity $k$, then every double coset $H g H$ such that $H^g \cap H$ is infinite, has a  representative of length at most $2 \nu(k)$. Notice that $|H^g \cap H| = \infty$ is equivalent to saying that $H^g \cap H \neq 1$ since $G$ is torsion-free, so it follows that $H$ satisfies the BNTCI property. Moreover, a set of BNTCI representatives $\mathcal{R}_G(H)$ in $G$ by $H$ can be found effectively.

Hence, again as in the free group case, we can use Lemma \ref{le:general_2} as the criterion for conjugacy stability.

Suppose $g \in \mathcal{R}_G(H)$ and $u \in g H g^{-1} \cap H$. Centralizers in $G$ are infinite cyclic, so we can assume that $C_H(u) = \langle u \rangle$. Repeating the argument given in Section \ref{se:free} for free groups, we obtain that if $C_G(u) = C_H(u)$, which happens only if $h$ is not a proper power of some element in $G$, then 
$$H \cap C_G(u) g = H \cap \langle u \rangle g \subseteq H \cap H g = \varnothing.$$
It follows that if the intersection $g H g^{-1} \cap H$ contains an element which is not a proper power if $G$, then $H$ fails to be conjugacy stable.

Suppose now all elements from $g H g^{-1} \cap H$ are proper powers in $G$. Then one can show that $g H g^{-1} \cap H$ is infinite cyclic generated by a proper power in $G$, as it happened for free groups. Indeed, suppose on the contrary that $H^{g^{-1}} \cap H$ is non-cyclic. Since $G$ is torsion-free, it is equivalent to saying that $H^{g^{-1}} \cap H$ is non-elementary. By \cite[Theorem C]{Kapovich:1999}, there exists a non-trivial subgroup $M \leqslant H^{g^{-1}} \cap H$ which is quasi-convex and malnormal in $G$. Now take an arbitrary non-trivial $a \in M$. Since $a \in H^{g^{-1}} \cap H$, it is a proper power of some element $x \in G \ssm \left( H^{g^{-1}} \cap H \right)$, in particular, $x \notin M$. But then $a \in M^x \cap M$, which contradicts malnormality of $M$. The contradiction comes from the assumption that $H^{g^{-1}} \cap H$ is non-cyclic.

Hence, as in the case of free groups, for every $g \in \mathcal{R}_G(H)$ we have $g H g^{-1} \cap H = \langle h_g^{n_g} \rangle$, where $h_g$ is a root element in $G$ and $n_g \geqslant 2$. To find out if $H$ is conjugacy stable we have to check if $H \cap C_G(h_g^{n_g}) g = H \cap C_G(h_g) g = H \cap \langle h_g \rangle g$ is non-empty. This can be summarized in the following proposition.

\begin{proposition}
\label{pr:hyp_algorithm}
Let $G$ be a torsion-free hyperbolic group. There exists an effective procedure to check if any quasi-convex subgroup $H \leq G$ is conjugacy stable or not. 
\end{proposition}

\begin{proof} Let $H$ be a quasi-convex subgroup of $G$.

\textbf{Step 1.} Compute $\mathcal{R}_G(H)$, which is a finite set, using \cite[Proposition 6.7(1)]{KMW:2014}.

\textbf{Step 2.} For every $g \in \mathcal{R}_G(H)$, find generators of the intersection $g H g^{-1} \cap H$. This can be done algorithmically since $H$ is quasi-convex (\cite[Theorem 4.1]{KMW:2014} and \cite[Proposition 6.10]{KMW:2014}).

\textbf{Step 3.} Check if every $g H g^{-1} \cap H$ is cyclic. This can be done effectively using the fact that a torsion-free hyperbolic group is commutative transitive, so we just need to check if $g H g^{-1} \cap H$ is an abelian subgroup by verifying if $[x, y] = 1$ in $G$ for every pair of generators $x, y$ of $g H g^{-1} \cap H$. All these operations can be done effectively, since the word problem is solvable for hyperbolic groups.

If at least one intersection $g H g^{-1} \cap H$ is not cyclic, or if it is cyclic, but the generator is not a proper power in $G$, then stop. In this case $H$ is not conjugacy stable. Notice that we can check if an element is a proper power in $G$ since we can effectively compute a generating set for its cyclic centralizer.

\textbf{Step 4.} Assuming that for every $g \in \mathcal{R}_G(H)$ we have 
$$g H g^{-1} \cap H = \langle h_g^{n_g} \rangle,$$
where $h_g$ is a root element in $G$ and $n_g \geqslant 2$, check if $H \cap \langle h_g \rangle g$ is non-empty.
Effective computation of the coset $\langle h_g \rangle g$ follows from \cite[Proposition 6.2]{KMW:2014}, the 
intersection $H \cap \langle h_g \rangle g$ is computable by \cite[Proposition 6.10]{KMW:2014}, and checking if it is non-empty is effective by \cite[Proposition 6.3]{KMW:2014}.

If $H \cap \langle h_g \rangle g$ is non-empty for every $g \in \mathcal{R}_G(H)$ then $H$ is conjugacy stable, otherwise it is not conjugacy stable.
\end{proof}

\subsection{Limit groups}
\label{se:limit}

Recall that a group $G$ is called {\em fully residually free} if for any finite set of non-trivial elements $g_1, \ldots, g_k\in G$ there exists a homomorphism $\phi : G \to F$, where $F$ is a free group, such that $\phi(g_i) \neq 1$ in $F$ for every $i \in [1, k]$. Finitely generated fully residually free groups are also known as {\em limit groups}, where the term ``limit'' comes from the definition based on sequences of homomorphisms to free groups and their stable kernels
(see, for example, \cite{Bestvina_Feighn_2009} for details). The class of limit groups drew a lot of attention because of its connections with equations over free groups and Tarski problems for free groups. Limit groups are known to be {\em toral relatively hyperbolic}. Namely, every limit group $G$ is a torsion-free group hyperbolic relative to a {\em peripheral structure} $\mathcal{P}$, which is a finite collection $P_1, \ldots, P_r$ of finitely generated free abelian subgroups of $G$ called {\em peripheral subgroups}. In this case, one can classify elements of $G$ with respect to the peripheral structure $\mathcal{P}$ as follows: $g \in G$ is called {\em parabolic} if it conjugates into a peripheral subgroup and {\em hyperbolic} otherwise (notice, that the case of {\em elliptic} elements is impossible since every toral relatively hyperbolic group is torsion-free). Similarly, a subgroup $H \leqslant G$ is called {\em parabolic} if it conjugates into a peripheral subgroup and {\em hyperbolic} otherwise (a hyperbolic subgroup contains a hyperbolic element). 

The definitions and notions from the theory of relatively hyperbolic groups mentioned above is basically all we need in what follows. Detailed exposition of the theory can be found in \cite{Farb:1998}, \cite{Bowditch:2012}, \cite{Drutu_Sapir:2005}, \cite{Osin:2006}.

Algorithmic problems for limit groups have been studied from different viewpoints. One approach is based on considering limit groups as finitely generated subgroups of $F^{\mathbb{Z}[t]}$ (see \cite{Kharlampovich_Myasnikov:1998a} for details) and then translating Stallings folding techniques from subgroups of free groups to subgroups of $F^{\mathbb{Z}[t]}$ (see \cite{Myasnikov_Remeslennikov_Serbin:2005, Kharlampovich_Myasnikov_Remeslennikov_Serbin:2006}).

\smallskip

Now, suppose $G$ is a limit group and $H \leqslant G$ is finitely generated. In order to take advantage of Lemma \ref{le:general_2} while deciding if $H$ is conjugacy stable we need the BNTCI property of $H$. Unfortunately, the situation is more complicated than in the cases of free and word hyperbolic groups, and the BNTCI property does not hold for finitely generated subgroups of $G$ in general. At the same time, the following results are known. 

\begin{theorem} \cite[Theorem 7]{Kharlampovich_Myasnikov_Remeslennikov_Serbin:2006}
\label{th:limit_1}
Let $G$ be a limit group and let $H \leqslant G$ be finitely generated. Then one can effectively find a finite family $\mathcal{J}_G(H)$ of non-trivial finitely generated subgroups of $G$ (given by finite generating sets), such that
\begin{enumerate}
\item every $J \in \mathcal{J}_G(H)$ is of one of the following types
$$H^{g_1} \cap H,\ \ \ H^{g_1} \cap C_H(g_2),$$
where $g_1 \in G \ssm H,\ g_2 \in H$ (the elements $g_1,\ g_2$ can be found effectively), and
\item for any non-trivial intersection $H^g \cap H,\ g \in G \ssm H$ there exists $J \in \mathcal{J}_G(H)$ and $f \in H$ such that
$$H^g \cap H = J^f,$$
where $J$ and $f$ can be found effectively.
\end{enumerate}
\end{theorem}
As was noted in the end of Section \ref{se:general}, the property BNTCI implies existence of such a family ${\mathcal J}_G(H)$, but the converse is not true. The following result is the maximum we can obtain in terms of the property BNTCI for finitely generated subgroups of limit groups.

\begin{theorem} \cite[Lemma 4.2]{Nikolaev_Serbin:2011}
\label{th:limit_2}
Let $G$ be a limit group and let $H \leqslant G$ be finitely generated. There are finitely many double cosets $H g_1 H, \ldots, H g_n H$ in $G$, where $g_1,\ldots, g_n$ can be found effectively, such that for any $g \in G$ if $H^g \cap H$ is non-abelian, then there exists $i \in [1,n]$ such that $g \in H g_i H$.
\end{theorem}

Actually, in \cite{Nikolaev_Serbin:2011}, the above result was formulated for finitely generated subgroups of $F^{\mathbb{Z}[t]}$, but it is equivalent to what we just stated.

So, it turns out that any finitely generated subgroup $H$ of a limit group $G$ ``almost'' has the BNTCI property - it holds only for non-abelian non-trivial intersections $H^g \cap H$. The following lemma shows that existence of a non-abelian intersection $H^g \cap H$ for some $g \in G$ is, actually, an obstacle for conjugacy stability of $H$.

\begin{lemma}
\label{le:limit_non-abelian-intersect}
Let $G$ be a limit group and $H \leqslant G$ be finitely generated. If $H^{g^{-1}} \cap H$ is non-abelian for some $g \in G \ssm H$, then $H$ is not conjugacy stable.
\end{lemma}
\begin{proof}
Fix a peripheral structure $\mathcal{P} = \{P_1, \ldots, P_r\}$ of finitely generated free abelian subgroups of $G$.

Assume that $H^{g^{-1}} \cap H$ is not abelian. We are going to show that in this case $H$ cannot be conjugacy stable. 

Notice that if $C_G(h) = C_H(h)$ for some $h \in H^{g^{-1}} \cap H$, then 
$$H \cap C_G(h) g = H \cap C_H(h) g \subseteq H \cap H g = \varnothing,$$
and $H$ fails to be conjugacy stable by Lemma \ref{le:general_1}. 

Now, assume that for every $h \in H^{g^{-1}} \cap H$ we have $C_H(h) \lneqq C_G(h)$. 

If $h \in H^{g^{-1}} \cap H$ is hyperbolic then the centralizer $C_G(h)$ is an infinite cyclic. Moreover, $h$ must be a proper power in $G$, otherwise $C_G(h) = C_H(h)$. If we assume that all elements of $H^{g^{-1}} \cap H$ are hyperbolic, take any $h_1, h_2, h_3 \in H^{g^{-1}} \cap H$ such that $h_1 h_2 = h_3$, so there exist $a, b, c \in G$ such that 
$$h_1 = a^k,\ h_2 = b^l,\ h_3 = c^m,\ k, l, m \geqslant 2$$
and
$$a^k b^l = c^m.$$
Combining the fact that $G$ (as well as $H$ and $H^{g^{-1}} \cap H$) is a limit group, that is, it is fully residually free, and the result of Lyndon and Schutzenberger \cite{Lyndon_Schutzenberger:1962}, we obtain that $a, b$ and $c$ commute with each other. Hence, $g H g^{-1} \cap H$ is cyclic - a contradiction with the assumption that $H^{g^{-1}} \cap H$ is not abelian. Hence, $H^{g^{-1}} \cap H$ must contain parabolic elements.

Finally, let $y \in H^{g^{-1}} \cap H$ be parabolic. Hence, there exists a peripheral subgroup $P_i$ such that $f^{-1} y f \in P_i$ for some $f \in G$. Since $H^{g^{-1}} \cap H$ is not abelian by assumption, there exists a hyperbolic element $x \in H^{g^{-1}} \cap H$, in particular, $f^{-1} x f \notin P_i$. By \cite[Lemma 3.5]{Arzhantseva_Minasyan_Osin:2007}, the product $(f^{-1} x f) (f^{-1} y^l f)$ is hyperbolic for infinitely many values of $l$. That is, $x y^l$ is hyperbolic for infinitely many values of $l$. Fix some value $l \geqslant 2$ (if all values of $l$ for which the product $(f^{-1} x f) (f^{-1} y^l f)$ is hyperbolic are negative, the argument is similar to what follows). Since all hyperbolic elements of $H^{g^{-1}} \cap H$ are proper powers in $G$, there exist $p, q \in G$ such that 
$$x = p^k,\ x y^l = q^m,\ k, m \geqslant 2$$
and 
$$p^k y^l = q^m,$$
where $k, l, m \geqslant 2$. Again, using the result of Lyndon and Schutzenberger one can conclude that $[p, y] = 1$, hence, $[x, y] = 1$ (since $G$ is commutative transitive). This gives a contradiction with the choice of $x$ and $y$. 

Hence, the assumption that for every $h \in H^{g^{-1}} \cap H$ we have $C_H(h) \lneqq C_G(h)$ leads to a contradiction and $H$ cannot be conjugacy stable.
\end{proof}

Obviously, Lemma \ref{le:limit_non-abelian-intersect} comes in handy by eliminating the difficulty of checking if $H \cap C_G(h) g$ is non-empty for infinitely many elements $h \in H^{g^{-1}} \cap H$. Indeed, if $H^{g^{-1}} \cap H$ is not abelian, then there are infinitely many non-commuting elements in $H^{g^{-1}} \cap H$, whose centralizers are different from each other, and effective application of Lemma \ref{le:general_1} becomes problematic.

Finally, we deal with abelian intersections.

\begin{lemma}
\label{le:abelian-intersect}
Let $G$ be a limit group and $H \leqslant G$ be finitely generated. Let $\mathcal{J}_G(H)$ be a finite collection of subgroups from Theorem \ref{th:limit_1} and assume that every $J \in \mathcal{J}_G(H)$ is abelian. Then $H$ is conjugacy stable if and only if for every $C, D \in \mathcal{J}_G(H)$ such that $w^{-1} C w = D$ for some $w \in G \ssm H$, the intersection $H \cap \overbar{C} w$ is non-empty, where $\overbar{C}$ is the maximal abelian subgroup of $G$ containing $C$.
\end{lemma}
\begin{proof}
Suppose $H$ is conjugacy stable and let $C, D \in \mathcal{J}_G(H)$ be such that $w^{-1} C w = D$ for some $w \in G \ssm H$. Hence, there exist $u \in C,\ v \in D$ such that $w^{-1} u w = v$, that is, $u \in H^{w^{-1}} \cap H$. By Lemma \ref{le:general_1}, the intersection $H \cap C_G(u) w$ is non-empty. Notice that since $C$ is abelian, $C_G(u)$ is a maximal abelian subgroup $\overbar{C}$ of $G$ that contains $C$, and it follows that $H \cap \overbar{C} w$ is non-empty.

\smallskip

Now we prove the converse. Assume that for every $C, D \in \mathcal{J}_G(H)$ such that $w^{-1} C w = D$ for some $w \in G \ssm H$, the intersection $H \cap \overbar{C} w$ is non-empty, where $\overbar{C}$ is the maximal abelian subgroup of $G$ containing $C$. We are going to show that $H$ is conjugacy stable.

Let $u, v \in H$ and $g \in G \ssm H$ be such that $g^{-1} u g = v$. Then $v \in H^g \cap H$ and by Theorem \ref{th:limit_1}, $H^g \cap H = D^h$ for some $D \in \mathcal{J}_G(H)$ and $h \in H$. Similarly, $H^{g^{-1}} \cap H$ is conjugate in $H$ to some $C \in \mathcal{J}_G(H)$, that is, $H^{g^{-1}} \cap H = C^f$, where $f \in H$. Since $u \in C^f,\ v \in D^h$ and $g^{-1} u g = v$, from the fact that $G$ is commutative transitive, it follows that $g^{-1} C^f g = D^h$, or $w^{-1} C w = D$, where $w = f g h^{-1}$. By assumption, the intersection $H \cap \overbar{C} w$ is non-empty, where $\overbar{C}$ is the maximal abelian subgroup of $G$ containing $C$. That is, there exists $w_0 \in H$ such that $w_0 = c w$ and $c \in \overbar{C}$. Observe that $c \in \overbar{C}$ implies that $c$ commutes with $f u f^{-1}$.
Finally, let $g_0 = f^{-1} w_0 h$. Notice that $g_0 \in H$ and we obtain
$$g_0^{-1} u  g_0 = (h^{-1} w_0^{-1} f)  u (f^{-1} w_0 h) = (h^{-1} w^{-1} c^{-1} f) u (f^{-1} c w h)$$
$$= (h^{-1} w^{-1} c^{-1}) (f u f^{-1}) (c w h) =  (h^{-1} w^{-1}) (f u f^{-1}) (w h)$$
$$ = (g^{-1} f^{-1}) (f u f^{-1}) (f g) = v.$$
In other words, $H$ is conjugacy stable.
\end{proof}

Finally, we are ready to present an algorithm that decides if a finitely generated subgroup of a limit group is conjugacy stable.

\begin{proposition}
\label{pr:limit_algorithm}
Let $G$ be a limit group. Then there exists an effective procedure to check if any finitely generated subgroup $H \leqslant G$ is conjugacy stable or not.
\end{proposition}
\begin{proof}
Let $H$ be a finitely generated subgroup of $G$. 

\smallskip

\textbf{Step 1.} Compute the set of subgroups $\mathcal{J}_G(H)$. The set is finite and all its elements can be found effectively (see \cite[Theorem 7]{Kharlampovich_Myasnikov_Remeslennikov_Serbin:2006}).

\textbf{Step 2.} Check if every $J \in \mathcal{J}_G(H)$ is abelian. If there is a non-abelian $J \in \mathcal{J}_G(H)$, then $H$ is not conjugacy stable by Lemma \ref{le:limit_non-abelian-intersect}.

\textbf{Step 3.} Assuming that every element of $\mathcal{J}_G(H)$ is abelian, for every $C,\ D \in \mathcal{J}_G(H)$ determine if $C$ is conjugate to $D$ in $G$ and if they are, compute some conjugating element $w_{C, D} \in G$. This can be done effectively (see \cite[Corollary 7]{Kharlampovich_Myasnikov_Remeslennikov_Serbin:2006}).

\textbf{Step 4.} For every $C,\ D \in \mathcal{J}_G(H)$ conjugate in $G$ (that is, $w_{C,D} \in G$ from the previous step exists), compute $H \cap \overbar{C} w_{C, D}$, where $\overbar{C}$ is the maximal abelian subgroup of $G$ containing $C$. Notice that $\overbar{C}$ can be found effectively by \cite[Theorem 8]{Kharlampovich_Myasnikov_Remeslennikov_Serbin:2006}: take a generator of $C$ and compute its centralizer in $G$. The intersection $H \cap \overbar{C} w_{C, D}$ can also be found effectively by \cite[Theorem 5]{Kharlampovich_Myasnikov_Remeslennikov_Serbin:2006}. Finally, according to Lemma \ref{le:abelian-intersect}, $H$ is conjugacy stable if and only if every such intersection $H \cap \overbar{C} w_{C, D}$ is non-empty which can be effectively determined.
\end{proof}

{\bf Acknowledgement.} We would like to thank Alexei Miasnikov for suggesting to take a look at the conjugacy stability problem in free groups. This initial study of ours was then extended and eventually resulted in this paper.

\end{document}